\documentclass[12pt, reqno]{amsart}
\usepackage{amsmath}
\usepackage{amsthm}
\usepackage{amssymb}
\usepackage[abbrev]{amsrefs}
\usepackage{mathrsfs}
\usepackage{yhmath}
\usepackage{bm}
\usepackage{enumitem}
\usepackage[all]{xy}
\usepackage{bbm}
\newtheorem{thm}{}[section]
\newtheorem{theorem}[thm]{Theorem}
\newtheorem{corollary}[thm]{Corollary}
\newtheorem{lemma}[thm]{Lemma}

\theoremstyle{definition}
\newtheorem{definition}[thm]{Definition}
\theoremstyle{remark}

\newtheorem{remark}[thm]{Remark}

\numberwithin{equation}{section}
\allowdisplaybreaks
\newcommand{\dem}{\ensuremath{\bm{\mu}}}

\newcommand{\Leb}{\ensuremath{\bm{L}}}

\newcommand{\Ind}{\ensuremath{\mathbbm{1}}}
\newcommand{\HB}{\ensuremath{\mathcal{H}}}
\newcommand{\Mt}{\ensuremath{\mathcal{M}}}
\newcommand{\Nt}{\ensuremath{\mathcal{N}}}
\newcommand{\ldf}{\ensuremath{\bm{\varphi_l}}}
\newcommand{\udf}{\ensuremath{\bm{\varphi_u}}}
\newcommand{\Disc}{\ensuremath{\mathbb{D}}}
\newcommand{\CC}{\ensuremath{\mathbb{C}}}
\newcommand{\RR}{\ensuremath{\mathbb{R}}}
\newcommand{\sss}{\ensuremath{\bm{s}}}
\newcommand{\ww}{\ensuremath{\bm{w}}}
\newcommand{\bb}{\ensuremath{\bm{b}}}
\newcommand{\ee}{\ensuremath{\bm{e}}}
\newcommand{\xx}{\ensuremath{\bm{x}}}
\newcommand{\yy}{\ensuremath{\bm{y}}}
\newcommand{\zz}{\ensuremath{\bm{z}}}
\newcommand{\hh}{\ensuremath{\bm{h}}}
\newcommand{\EE}{\ensuremath{\mathcal{E}}}
\newcommand{\Unt}{\ensuremath{\mathbb{E}}}
\newcommand{\BB}{\ensuremath{\mathcal{B}}}
\newcommand{\GG}{\ensuremath{\mathcal{G}}}
\newcommand{\NN}{\ensuremath{\mathbb{N}}}
\newcommand{\FF}{\ensuremath{\mathbb{F}}}
\newcommand{\XX}{\ensuremath{\mathbb{X}}}
\newcommand{\kk}{\ensuremath{\bm{k}}}
\newcommand{\Fou}{\ensuremath{\mathcal{F}}}
\newcommand{\SSB}{\ensuremath{\mathcal{S}}}
\newcommand{\Ts}{\ensuremath{\mathcal{T}}}
\newcommand{\UU}{\ensuremath{\mathcal{R}}}
\newcommand{\Sym}{\ensuremath{\mathbb{S}}}
\newcommand{\SL}{\ensuremath{\mathscr{L}}}
\newcommand{\XB}{\ensuremath{\mathcal{X}}}
\newcommand{\YB}{\ensuremath{\mathcal{Y}}}
\newcommand{\LL}{\ensuremath{\mathbb{B}}}
\newcommand{\YY}{\ensuremath{\mathbb{Y}}}
\newcommand{\VV}{\ensuremath{\mathbb{V}}}

\newcommand{\norm}[1]{\left\lVert #1\right\rVert}
\newcommand{\abs}[1]{\left\lvert #1\right\rvert}

\DeclareMathOperator{\spn}{span}
\DeclareMathOperator{\sgn}{sign}
\DeclareMathOperator*{\Ave}{Ave}

\AtBeginDocument{\def\MR#1{}}
\begin{document}
\title[Democracy of quasi-greedy bases in $p$-Banach spaces]{Democracy of quasi-greedy bases in $\bm p$-Banach spaces with applications to the efficiency of the Thresholding Greedy Algorithm in the Hardy spaces $\bm{H_{p}(\Disc^{d})}$}
\author[F. Albiac]{Fernando Albiac}
\address{Department of Mathematics, Statistics and Computer Sciences, and InaMat$^2$\\ Universidad P\'ublica de Navarra\\
Pamplona 31006\\ Spain}
\email{fernando.albiac@unavarra.es}

\author[J. L. Ansorena]{Jos\'e L. Ansorena}
\address{Department of Mathematics and Computer Sciences\\
Universidad de La Rioja\\
Logro\~no 26004\\ Spain}
\email{joseluis.ansorena@unirioja.es}

\author[G. Bello]{Glenier Bello}
\address{Mathematics Department\\
Universidad Aut\'onoma de Madrid\\
28049 Madrid\\
Spain}
\email{glenier.bello@uam.es}
\begin{abstract}
We use new methods, specific of non-locally convex quasi-Banach spaces, to investigate when the quasi-greedy bases of a $p$-Banach space for $0<p<1$ are democratic. The novel techniques we obtain permit to show in particular that all quasi-greedy bases of the Hardy space $H_p(\Disc)$ for $0<p<1$ are democratic while, in contrast, no quasi-greedy basis of $H_p(\Disc^d)$ for $d\ge 2$ is, solving thus a problem that was raised in \cite{AAW2021}. Applications of our results to other spaces of interest both in functional analysis and approximation theory are also provided.
\end{abstract}
\thanks{The first-named author acknowledges the support of the Spanish Ministry for Science and Innovation under Grant PID2019-107701GB-I00 for \emph{Operators, lattices, and structure of Banach spaces}. F. Albiac and J. L Ansorena were supported by the Spanish Ministry for Science, Innovation, and Universities under Grant PGC2018-095366-B-I00 for \emph{An\'alisis Vectorial, Multilineal y Aproximaci\'on}. G. Bello was supported
by National Science Centre, Poland grant UMO-2016/21/B/ST1/00241.}
\maketitle
\section{Introduction}\noindent
The formal development of a theory of greedy bases was spurred at the turn of the 21st century by the work of Konyagin and Temlyakov on the efficiency of the Thresholding Greedy Algorithm (TGA for short) in Banach spaces \cite{KoTe1999}. The TGA simply takes $m$ terms with maximum absolute values of coefficients from the expansion of a signal (function) relative to a fixed representation system (a basis). Different greedy algorithms originate from different ways of choosing the coefficients of the linear combination in the $m$-term approximation to the signal. Another name, commonly used in the literature for $m$-term approximation is \emph{sparse approximation}. Sparse approximation of functions is a powerful analytic tool which is present in many important applications to image and signal processing, numerical computation, or compressed sensing, to name but a few.

The simplicity in the implementation of the TGA and its connections with the geometry of the space attracted the attention of researchers with a more classical Banach space theory background and as a result, the last two decades have a seen a great progress in the functional analytic aspects of greedy approximation theory.

However, both from the abstract point of view of functional analysis as well as its applications, the development of a parallel theory of greedy bases for non-locally convex spaces was left out of the game, not because these spaces lack intrinsic interest but because of the absence of the foundational results that would cement this new ramification of the theory. It can be conceded that the locally convex case is more often used, especially in practical numerical computations, in large part due to the fact that convex algorithms are easy to implement. Nevertheless, there exist very well established scales of spaces in the non-locally convex setting that arise naturally in analysis and have been extensively studied, where it is necessary to do approximation theory. Take for instance the Hardy spaces of analytic functions on various domains in $\CC^{n}$ (see \cite{Duren1970}), Bergman spaces of analytic functions on various domains in $\CC^{n}$ (see \cite{HKZ1999}), Fefferman-Stein real Hardy spaces (see \cite{FS1972}), Besov, Sobolev, and Triebel-Lizorkin spaces (see \cite{Triebel2010}). Apart from those spaces, approximation theory in non-locally convex spaces plays a very important role in problems that arise in diverse areas such as approximation spaces, solutions of PDE's with data in Hardy spaces or Besov spaces that are not Banach spaces \cite{KMM}, approximation spaces and wavelet numerical methods \cite{Cohen1988}, layer potentials and boundary-value problems for second order elliptic operators with data in Besov Spaces, as in \cite{BM2016}. Approximation theory in non-locally convex spaces appears also naturally when studying interpolation problems even in the framework of Banach spaces. For example, the weak Lorentz space $L_{1,\infty}$ plays a key role in Marcinkiewicz interpolation theorem \cite{BennettSharpley1988}.

As was recently shown in \cite{AABW2021}, the main types of bases that are of interest in greedy approximation in the setting of Banach spaces, namely greedy, almost greedy, and quasi-greedy bases, are suitable as well for the use of the TGA in $p$-Banach spaces for $p<1$. The article \cite{AABW2021} was the springboard for subsequent research of different aspects related to the greedy algorithm in $p$-Banach spaces for $p<1$ (see \cites{AAW2021b, AAW2021}).
Our aim in this paper is to continue investigating the connection between quasi-greedy bases and their democracy functions in $p$-Banach spaces, initiated in \cite{AAW2021}, with an eye to the qualitative and quantitative study of the efficiency of the greedy algorithm in Hardy spaces $H_{p}$ and other important non-locally convex quasi-Banach spaces.

Both quasi-greedy and democratic bases were introduced by Konyagin and Temlyakov back in 1999 in their seminal paper \cite{KoTe1999}. These two special types of bases can nowadays be regarded in hindsight as the pillars that sustain the entire theory of greedy approximation in Banach spaces using bases. Certainly, apart from the pioneering characterization of greedy bases as those bases that are simultaneously unconditional and democratic from \cite{KoTe1999} this claim is supported by the subsequent characterization of almost greedy bases in Banach spaces as those bases that are at the same time quasi-greedy and democratic \cite{DKKT2003}.

Thus, although a priori, being almost greedy is more restrictive for a basis than being quasi-greedy, there exist spaces whose geometry forces quasi-greedy basis to be democratic. This is what happens for example with $c_0$, $\ell_2$, and $\ell_1$. The specific techniques used to prove that property in each space exhibit the critical structural aspect involved in making quasi-greedy basis be democratic. In $\ell_2$ the decisive ingredient is the fact that
its Rademacher type and cotype are $2$ (\cite{Woj2000}). In the space $\ell_1$ it is crucial that it is a GT-space \cite{DSBT2012}, while for $c_0$ what matters is that its dual is a GT-space \cite{DKK2003}.

The situation is different in $\ell_{p}$ when $p\in(1,2)\cup(2,\infty)$. Indeed, these spaces have unconditional bases (in particular, quasi-greedy bases) that are not democratic: To see this, one just has to remember that when $p\in(1,2)\cup(2,\infty)$, the space $\ell_{p}$ is isomorphic to $(\bigoplus_{n=1}^\infty\ell_2^n)_{\ell_p}$ \cite{Pel1960}, whose canonical basis is unconditional but not democratic.

The study of greedy-like bases in non-locally convex spaces sprang from the paper \cite{AABW2021}, where it is proved that the aforementioned characterizations of greedy bases and almost greedy bases remain valid in this more general framework. In this spirit, the authors of \cite{AAW2021} ventured out beyond the ``psychological" border of the index $p=1$ and proved that all quasi-greedy bases in the spaces $\ell_{p}$ for $0<p<1$ are democratic. The non-locally convex nature of these spaces required the introduction of new techniques in order to determine how the geometry of the space shapes the structure of their quasi-greedy bases. The results obtained in \cite{AAW2021} heavily rely on the theory of $\SL_p$-spaces for $0<p<1$ developed by Kalton in \cite{Kalton1984}.

In this paper we investigate the democracy of quasi-greedy bases in other classical non-locally convex quasi-Banach spaces. Our approach permits to obtain, for instance, that all quasi-greedy bases of the Hardy space $H_p(\Disc)$ for $0<p<1$ (as well as all quasi-greedy bases of its complemented subspaces) are democratic, solving thus in the positive Question 3.8 from \cite{AAW2021} in the case when $d=1$. As far as the Hardy spaces $H_p(\Disc^d)$ for $d\ge 2$ is concerned, Wojtaszczyk \cite{Woj2000} had shown that its canonical basis (which is unconditional) is not democratic. Here, our contribution consists of proving that, actually, no quasi-greedy basis of these multivariate Hardy spaces is democratic. Let us point out that the approach used to prove that quasi-greedy bases in $\ell_p$ are democratic falls short for the Hardy spaces, since the linear structure of the latter spaces is far more complex that the former. The new techniques that we had to develop to tackle the problem critically depend on the convexity properties of a quasi-Banach space regarded as a quasi-Banach lattice with the structure induced by its (unique) unconditional basis.

Our research suggests the pattern that if a quasi-Banach space $\XX$ (locally convex or otherwise) has a unique unconditional basis (up to equivalence and permutation) which is democratic (hence greedy) then all quasi-greedy bases of $\XX$ are democratic (hence almost greedy).

We also obtain valuable information about the democracy functions of quasi-greedy bases $\XB$ of other classical quasi-Banach spaces $\XX$ with a unique unconditional basis. For instance, we prove that the mixed-norm sequence spaces
$\ell_q\oplus\ell_p$ for $0<p<q<1$ have no almost greedy bases.

We close this introduction with a brief description of the contents of the paper. Section~\ref{sect:main} includes our advances in the theory of quasi-greedy bases. Previously, in Section~\ref{sect:preliminary} we set the terminology that we will use and gather some background results that we will need. In Section~\ref{sect:Hp} we provide applications to important spaces in functional analysis and approximation theory.

\section{Terminology and background}\label{sect:preliminary}\noindent
Throughout this paper we use standard facts and notation from Banach spaces and approximation theory (see, e.g., \cite{AlbiacKalton2016}). The reader will find the required specialized background and notation on greedy-like bases in quasi-Banach spaces in \cite{AABW2021}. Nonetheless, we record the notation that is most heavily used.

\subsection{Bases in quasi-Banach spaces}
Throughout this paper, a \emph{basis} of a quasi-Banach space $\XX$ over the real or complex field $\FF$ will be a norm-bounded countable family $\XB=(\xx_n)_{n\in \Nt}$ which generates the entire space $\XX$, and for which there is a (unique) norm-bounded family $\XB^*=(\xx_n^*)_{n\in \Nt}$ in the dual space $\XX^*$ such that $(\xx_n, \xx_n^*)_{n\in \Nt}$ is a biorthogonal system. A \emph{basic sequence} will be a sequence in $\XX$ which is a basis of its closed linear span. If $\XB=(\xx_n)_{n\in \Nt}$ is a basis, then it is \emph{semi-normalized}, i.e.,
\[
0<\inf_{n\in \Nt} \norm{\xx_n}\le \sup_{n\in \Nt} \norm{\xx_n}<\infty,
\]
and $\XB^*$ is a basic sequence called the \emph{dual basis} of $\XB$. Note that semi-normalized Schauder bases are a particular case of bases.

Given a linearly independent family of vectors $\XB=(\xx_n)_ {n\in \Nt}$ in $\XX$ and scalars $\gamma=(\gamma_n)_{n\in \Nt}\in\FF^\Nt$, we consider the map
\[
S_\gamma=S_\gamma[\XB,\XX]\colon \spn( \xx_n \colon n\in\NN) \to \XX,
\quad \sum_{n\in \Nt} a_n\, \xx_n \mapsto \sum_{n\in \Nt} \gamma_n\, a_n \, \xx_n.
\]
The family $\XB$ is an \emph{unconditional basis} of $\XX$ if and only if it generates the whole space $\XX$ and $S_\gamma$ is well-defined and bounded on $\XX$ for all $\gamma\in\ell_\infty$, in which case, the uniform boundedness principle yields
\begin{equation}\label{eq:lu}
K_{u}=K_{u}[\XB,\XX]:=\sup_{\norm{\gamma}_\infty\le 1} \norm{S_\gamma}<\infty.
\end{equation}
If $\XB$ is an unconditional basis, $K_u$ is called its \emph{unconditional basis constant.} Now, given $A\subseteq \NN$, we define the \emph{coordinate projection} onto $A$ (with respect to the sequence $\XB$) as
\[
S_A=S_{\gamma_A}[\XB,\XX],
\]
where $\gamma_A=(\gamma_n)_{n\in \Nt}$ is the family defined by $\gamma_n=1$ if $n\in A$ and $\gamma_n=0$ otherwise. It is known (see, e.g., \cite{AABW2021}*{Theorem 2.10}) that $\XB$ is an unconditional basis if and only if it generates $\XX$ and it is \emph{suppression unconditional}, i.e.,
\[
\sup \{\norm{S_A}\colon A\subseteq\NN \mbox{ finite}\}<\infty.
\]
Unconditional bases (indexed on the set $\NN$ of natural numbers) are a particular case of Schauder bases, and so semi-normalized unconditional bases are a particular case of bases.

\subsection{Quasi-greedy bases}
Given a basis $\XB=(\xx_n)_{n\in \Nt}$ of a quasi-Banach space $\XX$, with dual basis $\XB^*=(\xx_n^*)_{n\in \Nt}$, the \emph{coefficient transform}
\begin{equation*}\label{eq:Fourier}
\Fou\colon \XX \to \FF^\Nt , \quad f\mapsto (\xx_n^*(f))_{n\in \Nt}
\end{equation*}
is a bounded linear operator from $\XX$ into $c_0$. Thus, for each $m\in\NN$ there is a unique $A=A_m(f)\subseteq \NN$ of cardinality $\abs{A}=m$ such that whenever $i\in A$ and $j \in \NN\setminus A$, either $\abs{a_i}>\abs{a_j}$ or $\abs{a_i}=\abs{a_j}$ and $i<j$.
The \emph{$m$th greedy approximation} to $f\in\XX$ with respect to the basis $\XB$ is
\[
\GG_m(f)=\GG_m[\XB,\XX](f):=S_{A_m(f)}(f).
\]
Note that the operators $(\GG_m)_{m=1}^{\infty}$ defining the \emph{greedy algorithm} on $\XX$ with respect to $\XB$ are not linear nor continuous. The basis $\XB$ is said to be \emph{quasi-greedy} if there is a constant $C\ge 1$ such that
\[
\norm{\GG_m(f)}\le C\norm{f}, \quad f\in\XX, \, m\in\NN.
\]
Equivalently, by \cite{Woj2000}*{Theorem 1} (cf.\ \cite{AABW2021}*{Theorem 4.1}), these are precisely the bases for which the greedy algorithm converges, i.e.,
\[
\lim_{m\to\infty} \GG_{m}(f)=f, \quad f\in \XX.
\]

\subsection{Truncation quasi-greedy bases}
Another family of nonlinear operators of key relevance in the study of the greedy algorithm in a quasi-Banach space $\XX$ with respect to a basis $\XB=(\xx_n)_{n\in \Nt}$ is the sequence $(\UU_m)_{m=1}^{\infty}$ of restricted truncation operators defined as follows. Let
\[
\Unt=\{\lambda\in\FF \colon \abs{\lambda}=1\}.
\]
Given $A\subset\Nt$ finite and $\varepsilon=(\varepsilon_n)_{n\in A}\in\Unt^A$ we set
\[
\Ind_{\varepsilon,A}= \sum_{n\in A} \varepsilon_n \,\xx_n,
\]
and given $f\in\XX$ we define $\varepsilon(f)\in\Unt^{\Nt}$ by
\[
\varepsilon(f)=(\sgn(\xx_n^*(f)))_{n\in\Nt},
\]
where, as is customary, $\sgn(\cdot)$ denotes the sign function, i.e., $\sgn(0)=1$ and $\sgn(a)=a/\abs{a}$ if $a\in\FF\setminus\{0\}$. For $m\in \NN$, the $m$\emph{th-restricted truncation operator} $\UU_m\colon \XX \to \XX$ is the map
\[
\UU_m(f)=\min_{n\in A_m(f)} \abs{\xx_n^*(f)} \Ind_{\varepsilon(f),A_m(f)}, \quad f\in\XX,
\]
and the basis $\XB$ is said to be \emph{truncation quasi-greedy} if
\begin{equation}\label{eq:Cr}
C_r:=\sup_{m\in\NN}\norm{\UU_m}<\infty.
\end{equation}

For the sake of generality, most of our results below will be stated and proved for truncation quasi-greedy bases (or even for bases fulfilling weaker unconditionality conditions such as being UCC). However, the uneasy reader can safely replace ``truncation quasi-greedy basis'' with ``quasi-greedy basis'' and can rest assured of the validity of the corresponding statements thanks to the following theorem.

\begin{theorem}[\cite{AABW2021}*{Theorem 4.13}]\label{thm:QGBRTOP}
If $\XB$ is a quasi-greedy basis in a quasi-Banach space $\XX$ then $\XB$ is truncation quasi-greedy.
\end{theorem}

Semi-normalized unconditional bases are a special kind of quasi-greedy bases, and although the converse is not true in general, quasi-greedy bases always retain in a certain sense a flavour of unconditionality. For example, truncation quasi-greedy bases of quasi-Banach spaces are \emph{unconditional for constant coefficients} (UCC, for short) \cite{AABW2021}*{Proposition 4.16}. This means that there is a constant $C\ge 1$ such that $\norm{\Ind_{\varepsilon,A}}\le \norm{\Ind_{\varepsilon,B}}$ whenever $A$, $B$ are finite subsets of $\NN$ with $A\subseteq B$ and $\varepsilon\in\Unt^B$. If a basis $\XB$ is UCC then there is another constant $C_u\ge 1$ such that
\begin{equation}\label{eq:succ}
\norm{\Ind_{\delta,A}}\le C_u\norm{\Ind_{\varepsilon,A}}
\end{equation}
for all finite subsets $A$ of $\NN$ and all choices of signs $\delta$ and $\varepsilon\in\Unt^A$ (see \cite{AABW2021}*{Lemma 3.2}; for a detailed discussion on the unconditionality-related properties enjoyed by truncation quasi-greedy bases in quasi-Banach spaces we refer to \cite{AABW2021}*{Section 3}).

\subsection{Democracy functions}
Given a basis $\XB=(\xx_n)_{n\in \Nt}$ of a quasi-Banach space $\XX$ and $A\subseteq\Nt$ finite, we set $\Ind_A=\Ind_{\varepsilon,A}$, where $\varepsilon =1$ on $A$. The basis $\XB$ is said to be \emph{democratic} if there is a constant $D\ge 1$ such that
\[
\norm{\Ind_A}\le D \norm{\Ind_B}
\]
for any two finite subsets $A$ and $B$ of $\Nt$ with $\abs{A}\le \abs{B}$. The lack of democracy of a basis $\XB$ exhibits some sort of asymmetry. To measure how much a basis $\XB$ deviates from being democratic, we consider its \emph{upper democracy function}, also known as its \emph{fundamental function},
\[
\udf[\XB, \XX](m): = \udf(m)=\sup_{\abs{A}\le m}\norm{\Ind_A},\quad m\in\NN,
\]
and its \emph{lower democracy function},
\[
\ldf[\XB, \XX](m):= \ldf(m)=\inf_{\abs{A}\ge m}\norm{\Ind_A}, \quad m\in\NN.
\]
Notice that given a basis $\XB$ there is a constant $C_s$, depending only on the modulus of concavity of the space $\XX$, such that
\begin{equation}\label{eq:DemImp}
\norm{\sum_{n\in A} a_n\, \xx_n} \le C_s \udf[\XB,\XX](m),\quad \abs{A}\le m,\; \abs{a_n}\le 1.
\end{equation}

In the particular case that the basis $\XB$ is UCC the following hold:
\begin{enumerate}[label=(\roman*), leftmargin=*, widest=iii]
\item $\ldf(m)\lesssim\udf(m)$ for $m\in\NN$;
\item $\inf_{\abs{A}= m}\norm{\Ind_A}\lesssim \ldf[\XB, \XX](m)$ for $m\in\NN$; and
\item $\XB$ is democratic if and only $\udf(m)\lesssim\ldf(m)$ for $m\in\NN$, in which case it is \emph{super-democratic}, i.e., there is a constant $D\ge 1$ such that
\begin{equation*}
\norm{\Ind_{\varepsilon,A}} \le D \norm{\Ind_{\delta,B}}
\end{equation*}
for any two finite subsets $A$, $B$ of $\NN$ with $\abs{A}\le\abs{B}$, any $\varepsilon\in\EE^A$ and any $\delta\in\EE^B$.
\end{enumerate}
Here and throughout this paper, the symbol $\alpha_j\lesssim \beta_j$ for $j\in \Nt$ means that there is a positive constant $C$ such that the families of non-negative real numbers $(\alpha_j)_{j\in \Nt}$ and $(\beta_j)_{j\in \Nt}$ are related by the inequality $\alpha_j\le C\beta_j$ for all $j\in \Nt$. If $\alpha_j\lesssim \beta_j$ and $\beta_j\lesssim \alpha_j$ for $j\in \Nt$ we say $(\alpha_j)_{j\in \Nt}$ are $(\beta_j)_{j\in \Nt}$ are equivalent, and we write $\alpha_j\approx \beta_j$ for $j\in \Nt$.

\subsection{Quasi-Banach lattices}
Let $0<r\le\infty$. A quasi-Banach space $\XX$ is (topologically) \emph{$r$-convex} or \emph{$r$-normable} if there is a constant $C\ge 1$ such that
\begin{equation}\label{eq:pconvex}
\norm{\sum_{i\in A} f_i}\le C
\left(\sum_{i\in A}\norm{f_i}^r\right)^{1/r}
,\quad \abs{A}<\infty,\, f_i\in\XX.
\end{equation}
If a quasi-Banach space is $r$-convex then $r\le 1$. Conversely, by the Aoki-Rolewicz theorem any quasi-Banach space is $r$-convex for some $r\in(0,1]$. In turn, any $r$-convex quasi-Banach space $\XX$ becomes an $r$-Banach space under a suitable renorming, i.e., $\XX$ can be endowed with an equivalent quasi-norm satisfying \eqref{eq:pconvex} with $C=1$.

The existence of a lattice structure in $\XX$ leads to a (related but) different notion of convexity. A quasi-Banach lattice $\XX$ is said to be \emph{$r$-convex} ($0<r<\infty$) if there is a constant $C$ such that
\begin{equation}\label{eq:ConvexLattice}
\norm{\left(\sum_{i\in A}\abs{f_i}^r\right)^{1/r}}\le C \left(\sum_{i\in A}\norm{f_i}^r\right)^{1/r}, \quad \abs{A}<\infty,\; f_i\in\XX.
\end{equation}
where the lattice $r$-sum $
\left(\sum_{i\in A}\abs{f_i}^r\right)^{1/r}\in\XX$ is defined unambiguously exactly as for the case of Banach lattices (cf.\ \cite{LinTza1979}*{pp. 40--41} and \cite{Popa1982}). If $\XX$ is an $r$-convex quasi-Banach lattice, we will denote by $M^{(r)}(\XX)$ the smallest constant $C$ such that \eqref{eq:ConvexLattice} holds.

We will also consider lattice averages and use them to reformulate lattice convexity in those terms. Given $0<r\le \infty$, the $r$-average of a finite family $(f_i)_{i\in A}$ in a quasi-Banach lattice $\XX$ is defined as
\[
\left(\Ave_{i\in A} \abs{f_i}^r\right)^{1/r}=\left(\frac{1}{\abs{A}}\sum_{i\in A} \abs{f_i}^r \right)^{1/r}=\abs{A}^{-1/r} \left(\sum_{i\in A}\abs{f_i}^r\right)^{1/r}.
\]
This way, the quasi-Banach lattice $\XX$ is $r$-convex with $M^{(r)}(\XX)\le C<\infty$ if and only if
\[
\norm{\left(\Ave_{i\in A} (\abs{f_i}^r\right)^{1/r}} \le C \left(\Ave_{i\in A} (\norm{f_i}^r\right)^{1/r}, \quad \abs{A}<\infty,\; f_i\in\XX.
\]

Defining $r$-sums and $r$-averages in quasi-Banach lattices allows us to state a lattice-valued version of Khintchine's inequalities:

\begin{theorem}\label{lem:KintchineLattice}
Let $\XX$ be a quasi-Banach lattice. For each $0<r<\infty$ there are constants $T_r$ and $C_r$ such that for any finite family $(f_i)_{i\in A}$ in $\XX$,
\begin{equation*}
\frac{1}{C_r} \left(\Ave_{\varepsilon_i=\pm 1} \abs{\sum_{i\in A} \varepsilon_i\, f_i}^r\right)^{1/r}
\le \left(\sum_{i\in A} \abs{f_i}^2\right)^{1/2}
\le T_r \left(\Ave_{\varepsilon_i=\pm 1} \abs{\sum_{i\in A} \varepsilon_i\, f_i}^r\right)^{1/r}.
\end{equation*}
\end{theorem}

\begin{proof}
Just apply the functional calculus $\tau$ described in \cite{LinTza1979}*{Theorem 1.d.1} to the functions $f$, $g\colon\RR^n \to \RR$ given by
\[
f((x_i)_{i=1}^n)=\left(\sum_{i=1}^n \abs{x_i}^2\right)^{1/2},
\quad g((x_i)_{i=1}^n)= \left(\Ave_{\varepsilon_i=\pm 1} \abs{\sum_{i=1}^n \varepsilon_i\, x_i}^{r}\right)^{1/r},
\]
and use Khintchine's inequalities \cite{Kinchine1923}.
\end{proof}

If $\XX$ is a $r$-convex quasi-Banach lattice, then $\XX$ is a $\min\{r,1\}$-convex quasi-Banach space. The converse does not hold, to the extent that there are quasi-Banach lattices that are not $r$-convex for any $r>0$ (see \cite{Kalton1984b}*{Example 2.4}). Theorem 2.2 from the same paper characterizes quasi-Banach lattices with some nontrivial convexity as those that are $L$-convex. A quasi-Banach lattice $\XX$ is said to be \emph{$L$-convex} if there is $\varepsilon>0$ so that whenever $f$ and $(f_i)_{i\in A}$ in $\XX$ satisfy $0\le f_i\le f$ for every $i\in A$, and $(1-\varepsilon)\abs{A} f\le \sum_{i\in A} f_i$ we have $\varepsilon \norm{f }\le \max_{i\in A} \norm{f_i}$.

The aforementioned Theorem 2.2 from \cite{Kalton1984b} also gives the following.
\begin{theorem}\label{thm:ConvexInterval}
Let $\XX$ be a quasi-Banach lattice and let $0<r<\infty$. If $\XX$ is lattice $r$-convex, then it is lattice $s$-convex for every $0<s<r$.
\end{theorem}

Quoting Kalton from \cite{Kalton1984b}, $L$-convex lattices behave similarly to Banach lattices in many respects. The following result, which generalizes to quasi-Banach lattices \cite{Maurey1974}*{Lemme 5}, is in this spirit.
\begin{lemma}\label{lem:MaureyQB}
Let $\XX$ be an $L$-convex quasi-Banach lattice. Then for every $0<r<\infty$ there is constant $C$ such that
\[
\norm{\left(\sum_{i\in A} \abs{f_i}^2\right)^{1/2}}
\le C\left(\Ave_{\varepsilon_i=\pm 1} \norm{\sum_{i\in A} \varepsilon_i\, f_i}^r\right)^{1/r},\quad \abs{A}<\infty,\; f_i\in\XX.
\]
\end{lemma}
\begin{proof}
Since that mapping
\[
r\mapsto \left(\Ave_{\varepsilon_i=\pm 1} \norm{\sum_{i\in A} \varepsilon_i\, f_i}^r\right)^{1/r}
\]
is increasing, by Theorem~\ref{thm:ConvexInterval}, we can assume the $\XX$ is an $r$-convex Banach lattice. Then, by Theorem~\ref{lem:KintchineLattice},
\begin{align*}
\norm{\left(\sum_{i\in A} \abs{f_i}^2\right)^{1/2}}
&\le T_r \norm{\left(\Ave_{\varepsilon_i=\pm 1}\abs{\sum_{i\in A} \varepsilon_i\, f_i}^r\right)^{1/r}}\\
&\le T_r M^{(r)}(\XX) \left(\Ave_{\varepsilon_i=\pm 1} \norm{\sum_{i\in A} \varepsilon_i\, f_i}^r\right)^{1/r}
\end{align*}
for every finite family $(f_i)_{i\in A}$ in $\XX$.
\end{proof}

\begin{remark}
In light of Khintchine-Kahane-Kalton's inequalities (see \cite{Kalton1980}*{Theorem 2.1}) the index $r$ in Lemma~\ref{lem:MaureyQB} is irrelevant. From an opposite perspective, we point out that the proof of Lemma~\ref{lem:MaureyQB} does not depend on Khintchine-Kahane-Kalton's inequalities.
\end{remark}

We are interested in quasi-Banach lattices of functions defined on a countable set $\Nt$. The term \emph{sequence space} will apply to a quasi-Banach space $\LL\subseteq \FF^\Nt$ such that:
\begin{itemize}[leftmargin=*]
\item $\ee_n:=(\delta_{i,n})_{i\in \Nt}$ is a norm-one vector in $\LL$ for all $n\in \Nt$; and
\item if $f\in \XX$ and $g\in\FF^\Nt$ satisfy $\abs{g}\le\abs{f}$, then $g\in\XX$ and $\norm{g}\le \norm{f}$.
\end{itemize}
That way, $\LL$ becomes a quasi-Banach lattice with the natural order. In this particular case $r$-sums take a more workable form: given $f_i=(a_{i,n})_{n\in \Nt}$ for $i\in A$,
\[
\left(\sum_{i\in A}\abs{f_i}^r\right)^{1/r} = \left(\left( \sum_{i\in A} \abs{a_{i,n}}^r\right)^{1/r}\right)_{n\in \Nt}.
\]

There is a close relation between sequence spaces and unconditional bases. If $\LL$ is a sequence space, the unit vector system $\EE_\Nt=(\ee_n)_{n\in \Nt}$ is a $1$-unconditional basic sequence of $\LL$ whose coordinate functionals are the restriction to $\LL$ of the coordinate functionals $(\ee_n^*)_{n\in \Nt}$ defined as
\[
\ee_n^*(f)=a_n, \quad f=(a_n)_{n\in\Nt}\in\FF^\Nt.
\]
And, conversely, every semi-normalized unconditional basis $\XB$ of a quasi-Banach space $\XX$ becomes normalized and $1$-unconditional after a suitable renorming of $\XX$; this way we can associate a sequence space with $\XB$.

\subsection{Embeddings via bases and squeeze-symmetric bases}
Let $\XX$ be a quasi-Banach space with a basis $\XB=(\xx_n)_{n\in \Nt}$ and let $(\LL,\norm{\cdot}_\LL)$ be a sequence space on $\Nt$. Let us recall the following terminology, which we borrow from \cite{AlbiacAnsorena2016}.
\begin{enumerate}[label=(\alph*), leftmargin=*, widest=b]
\item We say that \emph{$\LL$ embeds in $\XX$ via $\XB$}, and put $\LL\stackrel{\XB}\hookrightarrow \XX$, if there is a constant $C$ such that for every $g\in\LL$ there is $f\in\XX$ such that $\Fou(f)=g$, and $\norm{f}\le C\norm{g}_{\LL}$.

\item We say that \emph{$\XX$ embeds in $\LL$ via $\XB$}, and put $\XX\stackrel{\XB}\hookrightarrow\LL$, if there is a constant $C$ such that $\Fou(f)\in \LL$ with $\norm{\Fou(f)}_{\LL}\le C \norm{f}$ for all $f\in \XX$.
\end{enumerate}

The sequence space $\LL$ is said to be \emph{symmetric} if $f_\pi:=(a_{\pi(n)})_{n\in \Nt}\in\LL$ and $\norm{f_\pi}_\LL=\norm{f}_\LL$ for all $f=(a_n)_{n\in \Nt}\in\LL$ and every permutation $\pi$ of $\Nt$.

Loosely speaking, by squeezing the space $\XX$ between two symmetric sequence spaces we obtain qualitative estimates on the symmetry of the basis in $\XX$. Thus, a basis $\XB$ is said to be \emph{squeeze-symmetric} if there are symmetric sequence spaces $\Sym_1$ and $\Sym_2$ on $\Nt$, which are close to each other in the sense that
\[
\norm{\Ind_A[\EE_\Nt,\Sym_1]} \approx \norm{\Ind_A[\EE_\Nt,\Sym_2]}, \quad \abs{A}<\infty,
\]
such that $\Sym_1\stackrel{\XB}\hookrightarrow \XX \stackrel{\XB}\hookrightarrow \Sym_2$. A basis of a quasi-Banach space is squeeze-symmetric if and only if it is truncation quasi-greedy and democratic (see \cite{AABW2021}*{Lemma 9.3 and Corollary 9.15}). In paticular, almost greedy bases are squeeze-symmetric.

Embeddings involving weighted Lorentz sequence spaces play an important role in greedy approximation theory using bases. For our purposes here, it will be sufficient to deal with weak Lorentz spaces.

Let $\ww=(w_n)_{n=1}^\infty$ be a weight, i.e., a sequence of nonnegative numbers with $w_1>0$, and let $\sss=(s_m)_{m=1}^\infty$ be its \emph{primitive weight} defined by $s_m=\sum_{n=1}^m w_n$. The \emph{weighted Lorentz sequence space} (on the countable set $\Nt$) $d_{1,\infty}(\ww)$ consists of all functions $f\in c_0(\Nt)$ whose non-increasing rearrangement $(a_m)_{m=1}^\infty$ satisfies
\[
\norm{f}_{\infty,\ww}:=\sup_{m\in\NN} s_m\, a_m<\infty.
\]

We must pay attention to whether $\sss$ is doubling, i.e., whether $\sss$ satisfies the condition
\[
s_{2m}\lesssim s_m, \quad m\in\NN.
\]

The $L$-convexity of the space will play a key role as well.
\begin{theorem}[\cite{CRS2007}*{Theorem 2.2.16} and \cite{AlbiacAnsorena2021}*{Theorem 6.1}]\label{thm:weakLorentzLconvex}
Let $\sss$ be the primitive weight of a weight $\ww$.
\begin{enumerate}[label=(\roman*), leftmargin=*, widest=ii]
\item The space $(d_{1,\infty}(\ww), \norm{\cdot}_{\infty,\ww})$ is quasi-normed if and only if $\sss$ is doubling. Moreover,
\item if $\sss$ is doubling, then $d_{1,\infty}(\ww)$ is an $L$-convex symmetric sequence space.
\end{enumerate}
\end{theorem}

Suppose $\XB$ is a truncation quasi-greedy basis of a quasi-Banach space $\XX$. Then, regardless of whether $\XB$ is democratic or not, the mere definition of lower democracy function yields a constant $C$ such that
\begin{equation}\label{eq:LorentzEmbed}
\sup_m a_m\, \ldf[\XB,\XX](m) \le C \norm{f}, \quad f\in\XX,
\end{equation}
where $(a_m)_{m=1}^\infty$ is the non-increasing rearrangement of $f$.

We point out that, since $\ldf[\XB,\XX]$ is not necessarily doubling (see \cite{Woj2014}), inequality \eqref{eq:LorentzEmbed} might not hold for an embedding of $\XX$ into a sequence space; thus we will need to appeal to the following consequence of \eqref{eq:LorentzEmbed}.

\begin{lemma}[see \cite{AABW2021}*{Corollary 9.13}]\label{lem:DemocracyEmdedding}
Let $\XB$ be a truncation quasi-greedy basis of a quasi-Banach space $\XX$. Let $\ww$ be a weight whose primitive weight $\sss=(s_m)_{m=1}^\infty$ is doubling. Then, $\XX \stackrel{\XB}\hookrightarrow d_{1,\infty}(\ww)$ if and only if $\sss \lesssim \ldf[\XB,\XX]$.
\end{lemma}

\subsection{Banach envelopes}
When dealing with a quasi-Banach space $\XX$ it is often convenient to know which is the `smallest' Banach space containing $\XX$. Formally, the \emph{Banach envelope} of a quasi-Banach space $\XX$ consists of a Banach space $\widehat{\XX}$ together with a linear contraction $J_\XX\colon\XX \to \widehat{\XX}$ satisfying the following universal property: for every Banach space $\YY$ and every linear contraction $T\colon\XX \to\YY$ there is a unique linear contraction $\widehat{T}\colon \widehat{\XX}\to \YY$ such that $\widehat{T}\circ J_\XX=T$. Pictorially we have
\[\xymatrix{
\widehat{\XX}\ar[drr]^{\widehat{T}} & \\
\XX \ar[u]^{J_\XX} \ar[rr]_T && \YY.
}\]
If a Banach space $\VV$ and a bounded linear map $J\colon \XX\to \VV$ are such that $\widehat{J}\colon \widehat{\XX}\to \VV$ is an isomorphism, we say that $\VV$ is an \emph{isomorphic representation} of the Banach envelope of $\XX$ via $J$.
For instance, given $0<p<1$, the Bergman space $A^1_{1-p/2}$ is an isomorphic representation of the Banach envelope of $H_p(\Disc)$ via the inclusion map (see \cites{DRS1969,Shapiro1976}).

If a quasi-Banach space $\XX$ has a basis $\XB$, and $\XX\stackrel{\XB}\hookrightarrow \ell_1$, then $\ell_1$ is an isomorphic representation of the Banach envelope of $\XX$ via the coefficient transform (\cite{AlbiacAnsorena2020}*{Proposition 2.10}). The following lemma is an immediate consequence of this and so we leave its verification to the reader.
\begin{lemma}\label{lem:factor}
Let $\XB$ be a basis of a quasi-Banach space $\XX$, and let $\LL$ be a sequence space. Suppose that $\XX\stackrel{\XB}\hookrightarrow \LL$ and that $\LL\subseteq\ell_1$ continuously. Then the envelope map $J_\XX\colon\XX\to\widehat{\XX}$ factors through the coefficient transform regarded as map from $\XX$ into $\LL$, i.e., there is a bounded linear map $J\colon\LL\to\widehat{\XX}$ such that $J_\XX=J\circ \Fou$.
\end{lemma}

\subsection{The Marriage Lemma}
A classical problem in combinatorics is to determine whether a given family $(\Nt_i)_{i\in I}$ of subsets of a set $\Nt$ admits a one-to-one map $\nu\colon I\to\Nt$ such that $\nu(i)\in \Nt_i$ for all $i\in I$. A necessary condition for the existence of such a map is
\begin{equation}\label{Hall:condition}
\abs{F}\le\abs{\bigcup_{i\in F} \Nt_i}, \quad F\subseteq I,\; \abs{F}<\infty.
\end{equation}
P.~Hall proved in \cite{Hall1935} that \eqref{Hall:condition} is also sufficient, provided the set $I$ of indices and all the sets $\Nt_i$ are finite. Subsequently, M.~Hall \cite{Hall1948} extended this result to the case when $I$ is not necessarily finite. Here, we will use a generalization by Wojtaszczyk \cite{Woj1997} of the latter result which was effectively used in his study of the uniqueness of unconditional bases in quasi-Banach spaces.

\begin{theorem}[see \cite{Woj1997}*{Corollary 3.1}]\label{thm:HKL}
Let $\Nt$ be a set and $(\Nt_i)_{i\in I}$ be a family of finite subsets of $\Nt$. Let $K\in\NN$. Suppose that
\[
\abs{F}\le K\abs{\bigcup_{i\in F} \Nt_i}
\]
for every $F\subseteq I$ finite. Then, there are a partition $(I_j)_{j=1}^K$ of $I$ and
one-to-one maps $\nu_j\colon I_j\to \Nt$ for $j=1$, \dots, $K$ such that $\nu_j(i)\in \Nt_i$ for each $i\in I$.
\end{theorem}

\section{Strongly absolute bases and their effect on the democracy functions of quasi-greedy basic sequences}\label{sect:main}\noindent
Loosely speaking, one could say that strongly absolute bases are `purely non-locally convex' bases, in the sense that if a quasi-Banach space $\XX$ has a strongly absolute basis then $\XX$ is far from being a Banach space. The strong absoluteness of a basis was identified and coined by Kalton et al.\ in \cite{KLW1990}, as the crucial differentiating feature of unconditional bases in quasi-Banach spaces. Here we work with a slightly different but equivalent definition.

\begin{definition}
An unconditional basis $\BB=(\bb_n)_{n \in \Nt}$ of a quasi-Banach space $\XX$ is \emph{strongly absolute} if for every constant $R>0$ there is a constant $C>0$ such that
\begin{equation}\label{eq:sb}
\sum_{n \in \Nt} \abs{\bb_n^*(f)} \norm{\bb_n}\le \max\left\{C \sup_{n \in \Nt} \abs{\bb_n^*(f)} \norm{\bb_n}, \frac{\norm{f}}{R}\right\}, \quad f\in\XX.
\end{equation}
\end{definition}

By definition, if we rescale a strongly absolute basis we obtain another strongly absolute basis. A normalized unconditional basis $\BB$ of $\XX$ is strongly absolute if $\XX\stackrel{\BB}\hookrightarrow\ell_1$, and $\XX$ is ``far from $\ell_1$,'' in the sense that whenever the quasi-norm of a vector in $\XX$ and the $\ell_1$-norm of its coordinate vector are comparable then the $\ell_{\infty}$-norm of its coordinates is comparable to both quasi-norms.
We refer to \cite{AlbiacAnsorena2021} for a list of quasi-Banach spaces with a strongly absolute unconditional basis. Some of those spaces will appear in Section~\ref{sect:Hp}, but for the time being we recall two different ways to find strongly absolute bases.

\begin{theorem}[\cite{AlbiacAnsorena2021}*{Proposition 6.5}]\label{thm:SACondition}
Let $\XX$ be a quasi-Banach space with a semi-normalized unconditional basis $\BB$. Suppose that
\[
\sum_{m=1}^\infty \frac{1}{\ldf[\BB,\XX](m)}<\infty.
\]
Then $\BB$ is strongly absolute.
\end{theorem}

\begin{theorem}[\cite{AlbiacAnsorena2021}*{Proposition 6.2}]\label{thm:weakLorentzSA}
Let $\ww$ be a weight whose primitive weight $\sss=(s_m)_{m=1}^\infty$ is doubling. Then the unit vector system of $d_{1,\infty}(\ww)$ is strongly absolute if and only if $\sum_{m=1}^\infty 1/s_m<\infty$.
\end{theorem}
We point out that the fundamental function of the unit vector system of $d_{1,\infty}(\ww)$ is the primitive weight of $\ww$. Hence, the `if' part of Theorem~\ref{thm:weakLorentzSA} can be directly derived from Theorem~\ref{thm:SACondition}.

Our goal in this section is to determine how the fact that a quasi-Banach space $\XX$ has a strongly absolute basis $\BB=(\bb_n)_{n\in\Nt}$ affects the democracy functions of quasi-greedy bases in $\XX$. To that end, if $\YB=(\yy_i)_{i\in\Mt}$ is another basis of $\XX$, we must estimate the size of $\norm{\sum_{i\in A}\yy_{i}}$ for any $A\subset \Mt$ finite in terms of the democracy functions of $\BB$. The following lemma from \cite{AAW2021} highlights the role played by the strongly absoluteness of the basis in that it makes possible picking large coefficients from the vectors of $\YB$.
The ``large coefficient'' technique was introduced by Kalton in \cite{Kalton1977} in his study of the uniqueness of unconditional basis in non-locally convex Orlicz sequence spaces.

Given $\delta>0$ and a finite family $\SSB=(\yy_i,\yy_i^*)_{i\in A}$ we consider the set of indices
\begin{equation*}
\Omega_\delta(\SSB,\BB)=\{n \in \Nt \colon \abs{\yy_i^*(\bb_n) \, \bb_n^*(\yy_i)}\ge \delta \mbox{ for some } i\in A\},
\end{equation*}
where $(\bb_n^{\ast})_{n\in \Mt}$ is the dual basis of $\BB$. The explicit definition of these sets goes back to the work of Wojtaszczyk on uniqueness of unconditional structure of quasi-Banach spaces \cite{Woj1997}.

\begin{lemma}[\cite{AAW2021}*{Lemma 3.3}]\label{lem:key}
Let $\LL$ be a quasi-Banach space with a strongly absolute basis $\BB$. Then, given $a\in(0,\infty)$ and $C\in(1,\infty)$, there is $\delta>0$ such that, whenever $\SSB=(\yy_i,\yy_i^*)_{i\in A}$ is a finite family in $\LL\times\LL^*$ with $\yy_i^*(\yy_i)=1$ and $ \norm{\yy_i}\norm{\yy_i^*}\le a$ for all $i\in A$, we have
\[
\abs{A}\le C \sum_{n\in \Omega_\delta(\SSB,\BB)} \abs{\sum_{i\in A} \yy_i^*(\bb_n) \, \bb_n^*(\yy_i)}.
\]
\end{lemma}

Since we will also come across with sets $\Omega_\delta(\SSB,\XB)$ associated to conditional bases $\XB$ of a quasi-Banach space $\XX$, in order to make headway we will need the following alteration of Lemma~\ref{lem:key}.

\begin{lemma}\label{lem:keydos}
Let $\XX$ be a quasi-Banach space with a basis $\XB=(\xx_n)_{n\in\Nt}$. Suppose that $\XX$ embeds via $\XB$ in a sequence space $\LL$ whose unit vector system is strongly absolute. Then, given $a\in(0,\infty)$ there are positive scalars $K$ and $\delta$ such that, whenever $\SSB=(\yy_i,\yy_i^*)_{i\in A}$ is a finite family in $\XX\times\XX^*$ with $\yy_i^*(\yy_i)=1$ and $ \norm{\yy_i}\norm{\yy_i^*}\le a$ for all $i\in A$, we have
\[
\abs{A}\le K \abs{\Omega_\delta(\SSB,\XB)}.
\]
\end{lemma}

\begin{proof}
The hypothesis implies that the space $\LL$ embeds continuously into $\ell_1$. Terefore, by Lemma~\ref{lem:factor}, there is a bounded linear map $J\colon\LL\to\widehat{\XX}$ such that $J(\Fou(f))=J_\XX(f)$ for all $f\in\XX$. Let $C_1$ denote its norm, and let $C_2$ be the norm of the coefficient transform $\Fou$ with respect to $\XB$, regarded as an operator from $\XX$ to $\LL$. Set
\begin{align*}
C_3&=\sup_{i\in A} \norm{\yy_i} \norm{\yy_i^*} <\infty,\; \mbox{ and}\\
C_4&=\sup_{n\in \Nt} \norm{\xx_n} \norm{\xx_n^*}<\infty.
\end{align*}
If we set
\[
\zz_i=\Fou(\yy_i)\in\LL\; \mbox{ and }\; \zz_i^*= \widehat{\yy_i^*}\circ J\in\LL^*,\quad i\in \Mt,
\]
then we have
\begin{align*}
\zz_i^*(\zz_i)&=\widehat{\yy_i^*}\circ J\circ \Fou(\yy_i)=\yy_i^*(\yy_i)=1\; \mbox{ and}\\
\norm{\zz_i} \norm{\zz_i^*} &\le \norm \Fou_{\XX\to\LL} \norm{\yy_i} \norm{J}\norm{\widehat{\yy_i^*}}
\le C_1 C_2 \norm{\yy_i} \norm{\yy_i^*}\le C_1C_2C_3
\end{align*}
for all $i\in A$. Moreover,
\begin{align*}
\ee_n^*(\zz_i)&=\ee_n^*(\Fou(\yy_i))=\xx_n^*(\yy_i)\; \mbox{ and}\\
\zz_i^*(\ee_n)&=\widehat{\yy_i^*}(J(\Fou(\xx_n)))=\widehat{\yy_i^*}(J_\XX(\xx_n))=\yy_i^*(\xx_n)
\end{align*}
for all $i\in A$ and $n\in\Nt$.

Set $a=C_1 C_2 C_3$ and an arbitrary $C>1$. We pick $\delta>0$ as in Lemma~\ref{lem:key} with respect to the unit vector system of $\LL$. We infer that
\[
\abs{A}\le C \sum_{n\in \Omega_\delta(\SSB,\XB)} \abs{\yy_i^*(\xx_n) \, \xx_n^*(\yy_i)}\le C C_3 C_4 \abs{\Omega_\delta(\SSB,\XB)}.\qedhere
\]
\end{proof}

The following elementary lemma puts an end to the auxiliary results of this preparatory section.
\begin{lemma}\label{lem:FiniteSets}
Let $\XB=(\xx_n)_{n\in\Nt}$ be a basis a quasi-Banach space $\XX$. The set
\[
\{n\in\Nt \colon \abs{f^*(\xx_n) \, \xx_n^*(f)}\ge \delta\}
\]
is finite for any given $f\in\XX$, $f^*\in\XX^*$, and $\delta>0$.
\end{lemma}

\begin{proof}
We need to prove that $(f^*(\xx_n)\,\xx_n^*(f))_{n\in\Nt}\in c_0(\Nt)$. But this follows from the facts that $(\xx_n^*(f))_{n\in\Nt}\in c_0(\Nt)$ and $(f^*(\xx_n))_{n\in\Nt}\in \ell_\infty(\Nt)$.
\end{proof}

The machinery developed above will permit us to obtain estimates for democracy functions of basic sequences in a quasi-Banach space that embeds in $\ell_1$ via a basis that is far from the canonical $\ell_1$-basis in a sense that will be made clear in place. We divide these estimates in two, depending on whether they involve lower or upper democracy functions.

\subsection{Lower estimates for democracy functions}
A subtle, yet important, obstruction to apply Lemma~\ref{lem:keydos} to basic squences $\YB$ in a quasi-Banach space $\XX$ is that the coordinate functionals associated to $\YB$ are not defined on $\XX$ but on the closed subspace of $\XX$ generated by $\YB$, denoted by $\YY$. If $\XX$ is locally convex, the Hahn-Banach theorem comes to our aid: any bounded linear functional on $\YY$ extends to a bounded linear functional on $\XX$ without increasing its norm. However, there are important spaces, such as Hardy spaces, that are not locally convex and so this extension cannot be taken for granted. This situation motivated Day \cite{Day1973} to define the \emph{Hahn-Banach Extension Property} (HBEP for short). We say that $\YY$ has the HBEP in $\XX$ if there is a constant $C$ such that for every $f^*\in\YY^*$ there is $g^*\in\XX^*$ such that $g^*|_\YY=f^*$ and $\norm{g^*} \le C\norm{f^*}$. Needless to say, $\XX$ has the HBEP in $\XX$. For the purposes of this paper, it suffices to keep in mind that any complemented subspace of a quasi-Banach space $\XX$ has the HBEP in $\XX$. We say that the basic sequence $\YB$ has the HBEP in $\XX$ if $\YY$ does.

The results in this section rely on the following lemma.

\begin{lemma}\label{lem:main}
Let $\XX$ be a quasi-Banach space with a basis $\XB=(\xx_n)_{n\in\Nt}$. Suppose that $\XX$ embeds in an $L$-convex sequence space $\LL$ via $\XB$ and that the unit vector system of $\LL$ is strongly absolute. Then, for every UCC basic sequence $\YB=(\yy_i)_{i\in \Mt}$ with the HBEP in $\XX$ there is a constant $c$ such that
\[
\ldf[\YB,\XX](m)\gtrsim \ldf[\EE_\Nt,\LL](\lceil cm\rceil), \quad m\in\NN.
\]
\end{lemma}

\begin{proof}
Choose for each $i\in\Mt$ an extension $\zz_i^*$ of the coordinate functional $\yy_i^*$ in such a way that $a=\sup_{i\in \Mt} \norm{\yy_i} \norm{\zz_i^*} <\infty$, and pick $K$ and $\delta$ as in Lemma~\ref{lem:keydos}.
For each $i\in \Mt$, let $f_i\in\FF^\Nt$ be given by
\[
f_i=\left( \zz_i^*(\xx_n) \, \xx_n^*(\yy_i)\right)_{n\in\Nt}.
\]
Notice that, for each $A\subseteq\Mt$ finite, $\max_{i\in A} \abs{f_i}\ge \delta$ on the set
\[
\Omega_A=\{n \in \Nt \colon \abs{\zz_i^*(\xx_n) \, \xx_n^*(\yy_i)}\ge \delta \mbox{ for some } i\in A\}.
\]
Fix $m\in\NN$. Let $A\subseteq \Mt$ be a finite set with $\abs{A}\ge m$. Using Lemma~\ref{lem:MaureyQB} we obtain
\begin{align*}
\norm{\Ind_A[\YB,\XX]}
&\approx \Ave_{\varepsilon_i=\pm1} \norm{\sum_{i\in A}\varepsilon_i\, \yy_i}\\
&\gtrsim \Ave_{\varepsilon_i=\pm1}\norm{\sum_{i\in A}\varepsilon_i\, \Fou(\yy_i)}_\LL\\
&\gtrsim \norm{\left(\sum_{i\in A}\abs{\Fou(\yy_i)}^2\right)^{1/2}}_\LL\\
&\gtrsim \norm{\left(\sum_{i\in A}\abs{f_i}^2\right)^{1/2}}_\LL\\
&\gtrsim \norm{\max_{i\in A}\abs{f_i}}_\LL\\
&\gtrsim \norm{\Ind_{\Omega_A}[\EE_\Nt, \LL]}.
\end{align*}
Now the statement follows since $\abs{\Omega_A}\ge \lceil m/K \rceil$.
\end{proof}

\begin{theorem}\label{thm:TQGU}
Let $\XX$ be a quasi-Banach space with a strongly absolute semi-normalized unconditional basis $\XB$ which induces an $L$-convex lattice structure on $\XX$. Suppose that $\YB$ is a UCC basic sequence with the HBEP in $\XX$. Then there is a constant $c>0$ such that
\[
\ldf[\YB,\XX](m) \gtrsim \ldf[\XB,\XX](\lceil cm\rceil), \quad m\in\NN.
\]
\end{theorem}

\begin{proof}
Just apply Lemma~\ref{lem:main} with $\LL$ the sequence space induced by the semi-normalized basis $\XB$.
\end{proof}

\begin{theorem}\label{thm:TQGTQG}
Let $\XX$ be a quasi-Banach space with a truncation quasi-greedy basis $\XB$. Let $\sss=(s_m)_{m=1}^\infty$ be a non-decreasing doubling sequence of positive scalars such that $ \ldf[\XB,\XX] \gtrsim \sss$ and
\[
\sum_{m=1}^\infty \frac{1}{s_m}<\infty.
\]
Suppose that $\YB$ is a UCC basic sequence with the HBEP in $\XX$. Then
\[
\ldf[\YB,\XX](m) \gtrsim s_m,\quad m\in\NN.
\]
\end{theorem}

\begin{proof}
By Theorem~\ref{thm:weakLorentzSA}, Theorem~\ref{thm:weakLorentzLconvex}, and Lemma~\ref{lem:DemocracyEmdedding} we can apply Lemma~\ref{lem:main} with $\LL=d_{1,\infty}(\ww)$, where $\ww$ is the weight with primitive weight $\sss$.
\end{proof}

\begin{corollary}\label{cor:UCCSQ}
Let $\XB$ be a squeeze-symmetric basis of a quasi-Banach space $\XX$. Suppose that
\[
\sum_{m=1}^\infty \frac{1}{\udf[\XB,\XX](m)}<\infty.
\]
Let $\YB$ is a UCC basic sequence with the HBEP in $\XX$. Then,
\[
\ldf[\YB,\XX]\gtrsim \udf[\XB,\XX]\approx \ldf[\XB,\XX].
\]
\end{corollary}

\begin{proof}
Notice that $\sss:=\udf[\XB,\XX]$ is doubling (see \cite{AABW2021}*{\S8}) and that, using democracy, $\ldf[\XB,\XX]\gtrsim\sss$. Hence, we can apply Theorem~\ref{thm:TQGTQG}.
\end{proof}

\subsection{Upper estimates for democracy functions}
Our results here heavily depend on the complementability of the closed subspaces of $\XX$ generated by the basic sequences we tackle. A basic sequence in $\XX$ that spans a complemented subspace is said to be a \emph{complemented basic sequence} of $\XX$.

\begin{lemma}\label{lem:UPF}
Suppose that a quasi-Banach space $\XX$ embeds via a basis $\XB=(\xx_n)_{n\in\Nt}$ in a sequence space $\LL$ whose unit vector system is strongly absolute. Let $\YB=(\yy_i)_{i\in\Mt}$ be a complemented basic sequence of $\XX$. If $\YB$ is truncation quasi-greedy then
\[
\udf[\YB,\XX]\lesssim \udf[\XB,\XX].
\]
\end{lemma}

\begin{proof}
Let $P\colon\XX\to\YY$ be a bounded linear projection, where $\YY$ is the closed subspace of $\XX$ generated by $\YB$. Set
\[
\zz_i^*=\yy_i^*\circ P, \quad i\in I,
\]
where $(\yy_i^*)_{i\in\Mt}$ in $\YY^*$ are the coordinate functionals of $\YB$. We have $\zz_i^*(\yy_i)=\yy_i^*(\yy_i)$ and $\norm{\zz_i^*} \norm{\yy_i } \le \norm{P} \norm{\yy_i^*} \norm{\yy_i }$ for every $i\in I$. Hence, by Lemma~\ref{lem:keydos}, there are $K\in\NN$ and $\delta>0$ such that, if we put
\[
\Omega_i=\{n \in \Nt \colon \abs{\zz_i^*(\xx_n) \, \xx_n^*(\yy_i)}\ge \delta\}, \quad i\in \Mt,
\]
then $\abs{A}\le K \abs{\cup_{i\in A} \Omega_i}$ for all $A\subseteq \Mt$ finite. Moreover, by Lemma~\ref{lem:FiniteSets}, $\abs{\Omega_i}<\infty$ for all $i\in\Mt$. Thus, by Theorem~\ref{thm:HKL}, there are a partition $(\Mt_j)_{j=1}^K$ of $\Mt$ and one-to-one maps $\nu_j\colon \Mt_j\to\Nt$ such that
\[
\abs{\zz_i^*(\xx_{\nu_j(i)}) \, \xx_{\nu_j(i)}^*(\yy_i)}\ge \delta, \quad 1\le j \le K, \; i\in \Mt_j.
\]

Pick $m\in\NN$, and let $A\subseteq \Mt$ be such that $\abs{A}\le m$. Set $A_j=A\cap\Mt_j$ for $j=1$, \dots, $K$. Set also
\[
a_{i,n}= \abs{\xx_n^*(\yy_i)} \sgn(\zz_i^*(\xx_n)), \quad i\in \Mt, \; n\in\Nt.
\]
For each $l\in A_j$ we have
\begin{align*}
\yy_l^* \left( P\left( \sum_{i\in A_j} a_{l,\nu_j(i)} \xx_{\nu_j(i)}\right)\right)
&=\sum_{i\in A_j} \abs{\zz_l^*(\xx_{\nu_j(i)}) \, \xx_{\nu_j(i)}^*(\yy_l)}\\
&\ge \abs{\zz_l^*(\xx_{\nu_j(l)}) \, \xx_{\nu_j(l)}^*(\yy_l)}\\
&\ge \delta.
\end{align*}
Consequently, if $C_u$ is as in \eqref{eq:succ} and $C_r$ is the truncation quasi-greedy constant of $\YB$ (see \eqref{eq:Cr}), we have
\[
\norm{\Ind_{A_j}[\YB,\YY] } \le \frac{C_u C_r}{\delta} \norm{P\left( \sum_{i\in A_j} \varepsilon_{l,\nu_j(i)} \xx_{\nu_j(i)}\right)}, \quad j=1,\dots, K.
\]
If we put $a=\sup_{n\in\Nt} \norm{\xx_n^*}$, $b=\sup_{i\in\Mt} \norm{\yy_i}$, and let $\kappa$ be the optimal constant such that
\[
\norm{\sum_{j=1}^K y_j} \le \kappa \sum_{j=1}^K \norm{y_j}, \quad y_j\in\YY,
\]
and $C_s$ be as in \eqref{eq:DemImp}, we obtain
\begin{align*}
\norm{\Ind_{A}[\YB,\YY] }
&=\norm{\sum_{j=1}^{K}\Ind_{A_j}[\YB,\YY]}\\
&\le\kappa\sum_{j=1}^{K}\norm{\Ind_{A_j}[\YB,\YY]}\\
&\le \kappa K\frac{C_uC_r}{\delta}\norm{P}abC_s\udf[\XB,\XX](m).\qedhere
\end{align*}
\end{proof}

\begin{theorem}\label{thm:UPF}
Let $\XX$ be a quasi-Banach space with a basis $\XB$. Assume that:
\begin{enumerate}[label=(\alph*), leftmargin=*, widest=a]
\item\label{it:UPFA} Either $\XB$ is unconditional and strongly absolute, or
\item\label{it:UPFB} there is a non-decreasing doubling sequence $\sss=(s_m)_{m=1}^\infty$ such that $ \ldf[\XB,\XX] \gtrsim \sss$ and
\[
\sum_{m=1}^\infty \frac{1}{s_m}<\infty.
\]
\end{enumerate}
Suppose that $\YB$ is a complemented truncation quasi-greedy basic sequence of $\XX$. Then \[
\udf[\YB,\XX]\lesssim \udf[\XB,\XX].
\]
\end{theorem}

\begin{proof}
If \ref{it:UPFA} holds, we let $\LL$ be the sequence space induced by the semi-normalized basis $\XB$. If \ref{it:UPFB} holds, we take $\LL=d_{1,\infty}(\ww)$, where $\ww$ is the weight with primitive weight $\sss$, and we appeal to Theorem~\ref{thm:weakLorentzSA} to claim that the unit vector system of $\LL$ is strongly absolute. This way, in both cases an application of Lemma~\ref{lem:UPF} yields the desired result.
\end{proof}

\section{Applications to quasi-greedy bases in Hardy spaces}\label{sect:Hp}\noindent
We are now in a position to apply the results of the previous section to the case when $\XX$ is the Hardy space $H_p(\Disc)$ for $0<p<1$. Our main result will be a ready consequence of the following theorem.

\begin{theorem}\label{thm:LUDF}
Let $\XX$ be a quasi-Banach space with a truncation quasi-greedy basis $\XB$ such that $\ldf[\XB,\XX]$ is doubling and
\[
\sum_{m=1}^\infty \frac{1}{\ldf[\XB,\XX](m)}<\infty.
\]
Suppose that $\YB$ is a complemented truncation quasi-greedy basic sequence of $\XX$. Then
\[
\ldf[\XB,\XX]\lesssim \ldf[\YB,\XX] \lesssim \udf[\YB,\XX]\lesssim \udf[\XB,\XX].
\]
\end{theorem}

\begin{proof}
Just combine Theorem~\ref{thm:TQGTQG} with Theorem~\ref{thm:UPF}.
\end{proof}

\begin{corollary}\label{cor:LUDF}
Let $\XB$ be a squeeze-symmetric basis of a quasi-Banach space $\XX$. Suppose that
\[
\sum_{m=1}^\infty \frac{1}{\udf[\XB,\XX](m)}<\infty.
\]
Let $\YB$ is be a complemented truncation quasi-greedy basis of $\XX$. Then $\YB$ is democratic, and
\[
\ldf[\XB,\XX]\approx \ldf[\YB,\XX] \approx \udf[\YB,\XX]\approx \udf[\XB,\XX].
\]
\end{corollary}

\begin{proof}
Since $\ldf[\XB,\XX]\approx \udf[\XB,\XX]$, $\ldf[\XB,\XX]$ is doubling, and
\[
\sum_{m=1}^\infty \frac{1}{\ldf[\XB,\XX](m)}<\infty.
\]
Thus, the result follows from Theorem~\ref{thm:LUDF}.
\end{proof}

\begin{corollary}
Let $0<p<1$. If $\YB$ is a truncation quasi-greedy basis of a complemented subspace of $H_p(\Disc)$ then $\YB$ is democratic with $\udf[\YB, H_p(\Disc)](m)\approx m^{1/p}$ for $m\in\NN.$ In particular, all quasi-greedy bases of $H_p(\Disc)$ are almost greedy.
\end{corollary}

\begin{proof}
The basis $\HB$ of $H_p(\Disc)$ constructed in \cite{Woj1984} is unconditional (hence truncation quasi-greedy) and democratic with
\[
\ldf[\HB,H_p(\Disc)](m) \approx m^{1/p}\approx \udf[\HB,H_p(\Disc)](m),\quad m\in \NN,
\]
(see \cite{KLW1990}). Now our claim follows from Theorem~\ref{cor:LUDF}.
\end{proof}

We note that for $d>1$ the canonical unconditional basis $\HB^d$ of the Hardy space $H_p(\Disc^d)$, constructed from the canonical unconditional basis $\HB$ of $H_p(\Disc)$ by means of tensor products, inherits the unconditionality from $\HB$, but not its democracy. Indeed, for every $d\in\NN$ we have
\[
\udf[\HB^d,H_p(\Disc^d)](m) \approx m^{1/p}, \quad m\in\NN,
\]
whereas
\begin{equation}\label{eq:HpLD}
\ldf[\HB^d,H_p(\Disc^d)](m) \approx h_m:=m^{1/p}(1+\log m)^{(d-1)(1/2-1/p)},\; m\in\NN,
\end{equation}
(see \cite{Temlyakov1998c}*{\S4} and \cite{Woj2000}*{\S4}).

\begin{theorem}\label{thm:LUDFBases}
Let $\XX$ be a quasi-Banach space with a truncation quasi-greedy basis $\XB$. Suppose that $\ldf[\XB,\XX]$ is doubling with
\[
\sum_{m=1}^\infty \frac{1}{\ldf[\XB,\XX](m)}<\infty.
\]
If $\YB$ is a truncation quasi-greedy basis of $\XX$, then
\[
\udf[\YB,\XX] \approx\udf[\XB,\XX].
\]
Moreover, if $\ldf[\YB,\XX]$ is doubling then
\[
\ldf[\YB,\XX] \approx\ldf[\XB,\XX].
\]
\end{theorem}

\begin{proof}
Apply Theorem~\ref{thm:LUDF} to obtain the estimates given there. Using Theorem~\ref{thm:UPF} with $\sss=\ldf[\XB,\XX]$ and the roles of $\XB$ and $\YB$ swapped, we deduce the equivalence between the upper democracy functions. If $\ldf[\YB,\XX]$ is doubling, we can apply Theorem~\ref{thm:TQGTQG} with $\sss=\ldf[\YB,\XX]$ and the roles of $\XB$ and $\YB$ swapped. This gives the equivalence between the lower democracy functions.
\end{proof}

\begin{theorem}\label{thm:LUDFBasesDos}
Let $\XX$ be a quasi-Banach space with a truncation quasi-greedy basis $\XB$. Suppose that $\ldf[\XB,\XX]$ is doubling with
\[
\sum_{m=1}^\infty \frac{1}{\ldf[\XB,\XX](m)}<\infty.
\]
If $\XB$ is not democratic then $\XX$ has no squeeze-symmetric bases.
\end{theorem}

\begin{proof}
Assume by contradiction that $\YB$ has a squeeze-symmetric basis $\YB$. Then $\ldf[\YB,\XX]\approx\udf[\YB,\XX]$ is doubling. By Theorem~\ref{thm:LUDFBases} we have $\udf[\YB,\XX] \approx\udf[\XB,\XX]$ and $\ldf[\YB,\XX] \approx\ldf[\XB,\XX]$, which leads to the absurdity $\ldf[\XB,\XX]\approx\udf[\XB,\XX]$.
\end{proof}

\begin{corollary}\label{thm:HpD}
Let $0<p<1$ and $d\in\NN$. If $d\ge 2$ the space $H_p(\Disc^d)$ has no sequeeze-symmetric bases. In particular, $H_p(\Disc^d)$ has no almost greedy bases.
\end{corollary}

\begin{proof}
Let $\hh=(h_m)_{m=1}^\infty$ be as in \eqref{eq:HpLD}. Since $\hh$ is doubling and non-equivalent to $(m^{1/p})_{m=1}^\infty$, the result follows from Theorem~\ref{thm:LUDFBasesDos}.
\end{proof}

We close this section with a quantitative estimate for the performance of the thresholding greedy algorithm implemented in Hardy spaces. To put this in context, we recall that in order to analyze the efficiency of the greedy algorithm with respect a basis $\XB$ of a quasi-Banach space $\XX$, it is customary to consider, for each $m\in\NN$, the smallest constant $C\ge 1$ such that
\[
\norm{f- \GG_m(f) } \le C \norm{f- z}
\]
for all $f\in \XX$ and all $m$-term linear combinations $z$ of vectors from $\XB$. This constant is called the $m$th \emph{Lebesgue constant}, and it is denoted by $\Leb_m[\XB,\XX]$. The growth of the Lebesgue constants of bases in Banach spaces has been studied in \cites{GHO2013,BBGHO2018}. We point out that the relation between the Lebesgue constants, the conditionality parameters, and the democratic deficiency parameters established in \cite{BBGHO2018} still holds in the non-locally convex setting. To be precise, if we put
\begin{equation*}
\kk_m[\XB,\XX] =\sup_{\abs{A}\le m} \norm{S_A[\XB,\XX]}, \quad m\in\NN,
\end{equation*}
and
\[
\dem_m[\XB,\XX]=\sup\limits_{\abs{A}=\abs{B}\le m}\frac{\norm{\Ind_A[\XB,\XX]}{\norm \Ind_B[\XB,\XX]}}, \quad m\in\NN
\]
then for any basis $\XB$ of a quasi-Banach space $\XX$ we have
\begin{equation}\label{eq:LebEstimates}
\Leb_m[\XB,\XX]\approx \max\{\kk_m[\XB,\XX] , \dem_m[\XB,\XX]\} , \quad m\in\NN.
\end{equation}
Moreover, if $\XB$ is UCC, then
\begin{equation}\label{dem:equivalence}
\dem_m[\XB,\XX]\approx\sup\limits_{l\le m}\frac{\udf[\XB,\XX](l)}{\ldf[\XB,\XX](l)}, \quad m\in\NN.
\end{equation}
(see \cite{AAB2021}*{\S1}).

\begin{corollary}\label{thm:HpD2}
Let $0<p<1$ and $d\in\NN$. Suppose $\XB$ is a truncation quasi-greedy basis of a complemented subspace of $H_p(\Disc^d)$. Then for $m\in\NN$ we have
\[
m^{1/p}(1+\log m)^{-\alpha} \lesssim \ldf[\XB,H_p(\Disc^d)](m) \lesssim \udf[\XB,H_p(\Disc^d)](m) \lesssim m^{1/p},
\]
where $\alpha=(d-1)(1/2-1/p)$. Consequently,
\[
\dem_m[\XB,H_p(\Disc^d)]\lesssim (1+\log m)^{\alpha},\quad m\in\NN.
\]
\end{corollary}

\begin{proof}
Just combine Theorem~\ref{thm:LUDF} with equations \eqref{eq:HpLD} and \eqref{dem:equivalence}.
\end{proof}

\begin{corollary}\label{cor:LOD}Let $0<p\le 1$ and $d\in\NN$. Suppose that $\XB$ is a truncation quasi-greedy basis of a complemented subspace of $H_p(\Disc^d)$. Then
\[
\Leb_m[\XB, H_{p}(\Disc^d) ] \lesssim (1+\log m)^{\alpha}, \quad m\in\NN,
\]
where $\alpha=\max\{1/p,(d-1)(1/2-1/p)\}$.
\end{corollary}

\begin{proof}
Combine \eqref{eq:LebEstimates}, Corollary~\ref{thm:HpD2}, and \cite{AAW2021b}*{Theorem 5.1}.
\end{proof}

\section{Further applications}\label{sect:applications}\noindent
Apart from the spaces $H_p(\Disc)$ and $\ell_p$ for $0<p<1$, Corollary~\ref{cor:LUDF} also applies to Fefferman-Stein's real Hardy spaces $H_p(\RR^d)$ for $d\in\NN$. More generally, applying Corollary~\ref{cor:LUDF} with $\XB$ a suitable wavelet basis gives that, if $0<p<1$, $0<p\le q\le\infty$, $s\in\RR$ and $d\in\NN$, all quasi-greedy bases of the homegeneous and the inhomegeneous Triebel-Lizorkin spaces $\ring{F}_{p,q}^s(\RR^d)$ and $F_{p,q}^s(\RR^d)$ and their complemented subspaces are democratic, with fundamental function of the same order as $(m^{1/p})_{m=1}^\infty$ (see \cite{AABW2021}*{\S11.3}).

Corollary~\ref{cor:LUDF} also applies to the $p$-convexified Tsirelson space $\Ts^{(p)}$, $0<p<1$.

The fundamental function of the unit vector system of the Lorentz sequence spaces $\ell_{p,q}$, $0< q\le \infty$, is equivalent to $(m^{1/p})_{m=1}^\infty$. We infer from Corollary~\ref{cor:LUDF} that all quasi-greedy bases of $\ell_{p,q}$ (we take its separable part if $q=\infty$) are democratic with fundamental function equivalent to $(m^{1/p})_{m=1}^\infty$.

An important class of sequence space whose unit vector system is not democratic are mixed-norm sequence spaces
\begin{equation}\label{eq:mixedspaces}
\ell_p\oplus\ell_q,\;
\ell_p(\ell_q),\;
\ell_q(\ell_p),\;
\left( \bigoplus_{n=1}^\infty \ell_q^n\right)_{\ell_p},\;
\left( \bigoplus_{n=1}^\infty \ell_p^n\right)_{\ell_q}, \quad 0<q<p.
\end{equation}
The lower democracy function of all these spaces is $(m^{1/p})_{m=1}^\infty$, whereas their upper lower democracy function is $(m^{1/q})_{m=1}^\infty$. In the case when $p<1$, we apply Theorem~\ref{thm:LUDFBasesDos} to obtain that no space from the list in \eqref{eq:mixedspaces} has an almost greedy basis. This partially solves \cite{AABW2021}*{Problem 13.8}.



\end{document}